\documentclass{amsart}
\usepackage{amsmath,amssymb}
\usepackage{graphicx}
\usepackage{epstopdf}
\usepackage{enumerate}
\usepackage{color}

\newtheorem{theorem}{Theorem}[section]
\newtheorem{proposition}[theorem]{Proposition}

\newtheorem{lemma}[theorem]{Lemma}

\theoremstyle{definition}
\newtheorem{definition}[theorem]{Definition}

\theoremstyle{remark}
\newtheorem{remark}[theorem]{Remark}

\newcommand{\diver}{\mathop{\rm div}\nolimits}

\newcommand{\mN}{{\mathbb N}}

 \newcommand{\cB}{\mathcal{B}}
 \newcommand{\cC}{\mathcal{C}}

\newcommand{\cF}{{\mathcal F}}

\newcommand{\cY}{{\mathcal Y}}
\newcommand\cZ{\mathcal Z}

\newcommand{\RR}{\mathbb{R}}
\DeclareMathOperator\curl{curl}

\begin{document}

\begin{abstract}
We characterize, using commuting zero-flux homologies, those volume-preserving vector fields on a $3$-manifold that are steady solutions of the Euler equations for some Riemannian metric. This result extends Sullivan's homological characterization of geodesible flows in the volume-preserving case. As an application, we show that the steady Euler flows cannot be constructed using plugs (as in Wilson's or Kuperberg's constructions). Analogous results in higher dimensions are also proved.
\end{abstract}

\title[A characterization of 3D steady Euler using commuting homologies]{A characterization of 3D steady Euler flows using commuting zero-flux homologies}

\thanks{}

\author{Daniel Peralta-Salas}
\address{Instituto de Ciencias Matem\'aticas, Consejo Superior de
Investigaciones Cient\'\i ficas, 28049 Madrid, Spain}
\email{dperalta@icmat.es}

\author{Ana Rechtman}
\address{IRMA, Universit\'e de Strasbourg, 7 rue Ren\'e Descartes,
  67084 Strasbourg, France}
\email{rechtman@math.unistra.fr}

\author{Francisco Torres de Lizaur}
\address{Max Planck Institute for Mathematics, Vivatsgasse 7, 53111 Bonn, Germany}
\email{ftorresdelizaur@mpim-bonn.mpg.de}

\date{}

\maketitle

\section{Introduction and main theorem}

The dynamics of an inviscid and incompressible fluid flow on a Riemannian\linebreak $3$-manifold $M$ is described by the Euler equations:
\begin{equation*}
\partial_t X+ \nabla_X X=-\nabla P\,, \;\; \diver X=0\,,
\end{equation*}
where $X$ is the velocity field of the fluid (which is a non-autonomous vector field on $M$) and $P$ is the pressure function, which is uniquely defined by the equations up to a constant. The operator $\nabla_X$ denotes the covariant derivative of a vector field along $X$ and $\diver$ is the divergence operator computed with respect to the Riemannian volume form.

When the vector field $X$ does not depend on time, we say that it is a stationary solution of the Euler equations, which models a fluid flow in equilibrium. It is well known~\cite{AKh, PS16} that the stationary Euler equations can be equivalently written as
\begin{equation*}
X\times \curl X=\nabla B\,, \;\; \diver X=0\,,
\end{equation*}
where $B:=P+\frac12|X|^2$ is the Bernoulli function of the fluid. We recall that given a metric $g$ and a volume form $\mu$, the $\curl$ operator and the vector product $\times$ are defined as
\[
i_{\curl X} \mu= d(i_X g) \,\,,\,\, i_{X\times Y} g=i_{Y} i_X \mu \,,
\]
where $i_Z g$ denotes the $1$-form dual to the vector field $Z$ using the metric.

An important milestone in the study of the stationary Euler flows, which marked the birth of the modern Topological
Hydrodynamics, is Arnold's structure theorem~\cite{AKh}. Roughly speaking, it shows that, when $X$ and $\curl X$ are not collinear, they behave as integrable Hamiltonian systems with 2 degrees of freedom. For extensions of this theorem to higher dimensions see~\cite{GK94}. The ``degenerate'' case corresponds to the so called \emph{Beltrami flows}, which are defined as those divergence-free vector fields such that $\curl X$ is proportional to $X$ (via a not necessarily constant proportionality factor).

The geometric wealth of the steady Euler flows has been unveiled in the last years and its study has attracted the attention of many people. Etnyre and Ghrist, developing an idea suggested by Sullivan, showed the equivalence between Reeb flows of a contact form and non-vanishing Beltrami fields with constant proportionality factor~\cite{EG00,EG02}; the case of general Beltrami fields corresponds to volume-preserving geodesible flows (or Reeb flows of stable Hamiltonian structures), as noticed by Rechtman~\cite{Anaphd,Re10}. More recently, Cieliebak and Volkov~\cite{CV17} constructed steady Euler flows that are not geodesible, and Izosimov and Khesin~\cite{IK17} characterized the vorticity functions of 2-dimensional steady Euler flows using Reeb graphs. Nevertheless, we are still far from having a deep understanding of the space of stationary solutions to the Euler equations.

Our goal in this paper is to give a complete characterization of
steady Euler flows {\it \`a la Sullivan} using zero-flux $2$-chains for a vector field that commutes with $X$. This contains, as a particular case, Sullivan's characterization of geodesible flows in terms of tangent homologies~\cite{Sullivan2} (for volume-preserving fields). To this end, we introduce the definition of an~\emph{Eulerisable flow}:

\begin{definition}[Eulerisable flow]\label{D:Euler} Let $M$ be a $3$-dimensional manifold endowed with a volume form $\mu$.  We say that a volume-preserving vector field $X$ is \emph{Eulerisable} if there is a Riemannian metric $g$ on $M$ (not necessarily inducing the volume form $\mu$) for which $X$ is a stationary solution to the Euler equations
\begin{equation*}
X \times \curl X=\nabla B\,, \qquad L_X \mu=0\,,
\end{equation*}
for some (Bernoulli) function $B : M \rightarrow \RR$. Here, $L_X$ denotes the Lie derivative. Equivalently, a volume-preserving vector field $X$ is Eulerisable if there is a Riemannian metric $g$ such that
$$i_Xd\alpha=-dB\,,$$
where $\alpha:=i_Xg$ is the $1$-form dual to $X$.
\end{definition}

\begin{remark}
When the Riemannian volume form $\mu_g$ does not coincide with the volume form $\mu$, the Euler equations presented above describe the behavior of an ideal \emph{barotropic} fluid, i.e., a fluid whose density is a function of the pressure. Indeed, writing $\mu=\rho\, \mu_g$, $\rho>0$ is a function which plays the role of the density. Then the vector field $X$ solves the equation $\rho \nabla_XX=-\rho\nabla p=:-\nabla \hat p$ provided that $\rho$ is a function of $p:=B-\frac12 \alpha(X)$, and $L_{\rho X}\mu_g=0$. In fact, if $X$ is a non-vanishing Eulerisable flow, the proof of the implication $(ii) \Rightarrow (i)$ of Theorem~\ref{T:main} below shows that the metric $g$ can be taken to be compatible with $\mu$, i.e. the Riemannian volume form $\mu_g$ coincides with $\mu$.
\end{remark}

To state our main theorem let us first introduce some notation. For a vector field $X$, we denote by $\cF_X$ the set of the boundaries of zero-flux $2$-chains, i.e.
$$
\cF_X =\Big\{\partial c \,:\, c \, \text{is a 2-chain with $\int_{c} i_X \mu=0$}\Big\}\,.
$$
Let $\mathcal{Z}_{X}$ and $\mathcal{C}_X$ be the cone of foliation currents and of foliation cycles of the vector field $X$, respectively (see Appendix~\ref{S:appendix} for a short review on the theory of currents). We recall that a \emph{foliation current} of a vector field $X$ is a \mbox{$1$-current} that can be approximated arbitrarily well (in the weak topology) by tangent $1$-chains. Equivalently, a foliation current can be approximated by \mbox{$1$-currents} of the form
\[
\sum_{i=1}^N c _i\, \delta^{p_i}_X\,,
\]
with $N \in \mathbb{N}$, $c_i \in [0, \infty)$ and $p_i \in M$. Recall that for any $p\in M$ the $1$-current $\delta^{p}_X$ is defined as
\[
\delta^{p}_X(\alpha)=\alpha_{p}(X) \text{ for any $1$-form $\alpha$} \, .
\]
A \emph{foliation cycle} is a closed foliation current, i.e., a foliation current whose kernel contains the linear subspace of exact $1$-forms.

All along this paper we say that a manifold is closed if it is compact and without boundary. The following is our main result:

\begin{theorem}\label{T:main}
Let $X$ be a non-vanishing volume-preserving vector field on a closed 3-manifold $M$ with trivial first cohomology group, $H^1(M)=0$. The following properties are equivalent:
\begin{enumerate}[(i)]
\item $X$ is Eulerisable.
\item There exists a $1$-form $\alpha$ such that $\alpha(X)>0$ and $\iota_Xd\alpha$ is closed.
\item There exists a (non-identically zero) vector field $Y$ that commutes with $X$, i.e. $[X,Y]=0$, such that no sequence of elements in $\cF_Y$ can approximate a non-trivial foliation cycle of $X$, that is, $\overline{\cF_Y}\cap
  \cC_X =\{0\}$.
\end{enumerate}
\end{theorem}

The proof of Theorem~\ref{T:main} yields the following characterization (new to the best of our knowledge) of Reeb fields in contact geometry, which is of independent interest. Notice that in its statement we do not need to assume that $X$ is volume-preserving.

\begin{theorem}\label{T:reeb}
Let $M$ be a closed 3-manifold which is not a fibration over the circle. Then a vector field $X$ is the Reeb field of a contact structure (possibly after rescaling) if and only if there is a smooth volume form $\mu$ such that no sequence of elements in $\cF_X$ can approximate a non-trivial foliation cycle of $X$, that is, $\overline{\cF_X}\cap
  \cC_X =\{0\}$.
\end{theorem}

We prove Theorems~\ref{T:main} and~\ref{T:reeb} in Sections~\ref{S:proof} and~\ref{S:Reeb}, respectively. Following Sullivan~\cite{Sullivan2}, key to the proofs is the Hahn-Banach theorem, which allows us to produce functionals that separate certain subsets of $1$-currents. The main application of Theorem~\ref{T:main}, which is presented in Section~\ref{S:plugs}, is to show that Eulerisable flows cannot be constructed using plugs. This proves, in particular, that an Eulerisable flow cannot contain plugs exhibiting Reeb components, as in Wilson's construction~\cite{Wilson,Ku96}. The fact that the steady Euler flows cannot contain Reeb components (not necessarily associated to a plug) was first observed in~\cite{CV17}. Finally, in Section~\ref{S:N} we prove a higher dimensional analog of Theorem~\ref{T:main} and its application showing that the steady Euler flows cannot exhibit plugs in arbitrary dimensions. For the reader's convenience, we have included Appendix~\ref{S:appendix} that contains a short introduction to the aspects of the theory of currents that are used in the proofs.

\section{Proof of the main theorem}\label{S:proof}

We will first establish the equivalence between items $(i)$ and $(ii)$ and then between items $(ii)$ and $(iii)$. In the proof, the volume form preserved by $X$ is denoted by $\mu$.

\subsection*{$(i) \Rightarrow (ii)$.} This is straightforward: let $g$ be a metric for which $X$ is a stationary solution of the Euler equations and define the dual $1$-form $\alpha:=i_{X}g$. Then $\alpha(X)>0$ and, since $X$ is an Euler flow, $i_X d\alpha$ is exact.

\subsection*{$(ii) \Rightarrow (i)$.} Let $\xi$ be the 2-plane distribution defined by the kernel of $\alpha$. It is standard that we can endow $\xi$ with a smooth Riemannian metric $g_{\xi}$, which can be extended trivially to a smooth degenerate quadratic form on $TM$, satisfying $g_{\xi}(X, \cdot)=0$. Then,
\[
g:= \frac{1}{\alpha(X)} \alpha\otimes\alpha +g_{\xi}
\]
is a metric on $M$ and it clearly verifies that $i_{X}g=\alpha$. Since $i_Xd\alpha$ is closed (and hence exact because $H^1(M)=0$), $X$ is a solution to the Euler equations with the metric $g$ and the volume form $\mu$. Notice that the freedom on the quadratic form $g_\xi$ allows one to multiply it by an appropriate (nonconstant) factor so that the Riemannian volume form induced by $g$ is the same as $\mu$.

\subsection*{$(ii) \Rightarrow (iii)$.} Let $Y$ be the vector field defined as $i_{Y} \mu:=d \alpha$.  It is easy to check that, since $X$ is volume-preserving, one has
\[
i_{[X, Y]} \mu=L_{X}  i_{Y} \mu \,.
\]
Since $d i_{Y} \mu=0$, we have that $Y$ preserves the volume $\mu$ as well, and also
\[
L_{X}  i_{Y} \mu=d i_{X} i_{Y} \mu=d i_{X} d \alpha=0 \,,
\]
thus implying that $X$ and $Y$ commute. It just remains to be proven that $\overline{\cF_Y}\cap \cC_X =\{0\}$. Suppose it is not the case, i.e., that there is a sequence of 2-chains $c_n$, with $c_n(i_Y \mu)=0$ and a non-zero foliation cycle $b$ satisfying
\[
\underset{n \rightarrow \infty}{\text{lim }}\partial c_n(\beta)=b(\beta)
\]
for any $1$-form $\beta$.
Now, for the $1$-form $\alpha$ we have on  one hand that
\[
\partial c_n(\alpha)= c_n (d\alpha)=c_n(i_Y \mu)=0\,,
\]
and on the other hand that
\[
b(\alpha)=\int_{b} \alpha>0
\]
because $\alpha(X)>0$ and $b$ is a foliation cycle of $X$. We arrive to a contradiction. Accordingly, $\overline{\cF_Y}\cap \cC_X =\{0\}$.

\subsection*{$(iii) \Rightarrow (ii)$.} Let $\cZ^1$ denote the vector space of $1$-currents on $M$, that is, the continuous dual of the space $\Omega^{1}$ of smooth $1$-forms on $M$ (see Appendix~\ref{S:appendix}).

It is well known \cite{Sullivan} that the set of foliation $1$-currents $\cZ_X \subset \cZ^1$ is a closed convex cone with compact convex base, i.e. there is a compact convex set $K \subset \cZ_X\setminus\{0\}$ such that
$$
\cZ_X=\{\lambda K \,,\, \lambda \in[0, \infty)\}\,.
$$
By item $(iii)$, since $\cZ_X\cap \overline{\cF_Y}=\cC_X\cap \overline{\cF_Y}$, we conclude that $K$ cannot intersect $\overline{\cF_Y}$. Since $\overline{\cF_Y}$ is a closed vector subspace of $\cZ^1$, a standard application of the Hahn-Banach theorem (see e.g.~\cite[Chapter 4, Theorem 4.5]{Simon}) ensures the existence of a continuous linear functional $\mathfrak{F}: \cZ^1 \rightarrow \RR$ that is strictly positive in $K$ (thus strictly positive in $\cZ_X \setminus \{0\}$ as well) and vanishes in $\overline{\cF_Y}$.

The continuous dual of $\cZ^1$ being $\Omega^{1}$, we can identify the continuous linear functional $\mathfrak{F}$ with a $1$-form $\alpha$; this form verifies, on the one hand, that
\begin{equation}\label{1}
\alpha(X)>0\,,
\end{equation}
because $b(\alpha) >0$ for any $b \in \cZ_X \setminus \{0\}$ and on the other hand
\begin{equation}\label{2}
\partial c(\alpha)=c(d \alpha) =0
\end{equation}
for any 2-current $c$ such that  $\partial c\in \overline{\cF_Y}$. In particular, $i_Yd\alpha=0$ because any \mbox{$2$-current} $c'$ tangent to $Y$ can be approximated by a sequence of $2$-chains tangent to $Y$ (and hence of zero flux), thus implying that $\partial c'\in \overline{\cF_Y}$ and so $c'(d\alpha)=0$.

Finally, let us show that $i_{X} d \alpha$ is closed and hence exact because $M$ is assumed to have trivial first cohomology group. Consider the linear subspace of $2$-forms that are proportional to $i_Y\mu$:
\[
\cY:=\{\omega \in
\Omega^2 \,,\, \omega=t i_Y \mu \,,\, t\in \RR\}\,.
\]
We claim that $d \alpha=  T i _Y \mu$ for some constant $T\neq 0$. This implies that $Y$ is volume-preserving and, since $X$ and $Y$ commute, it readily follows that $ d(i_X d\alpha)=0$.

Indeed, first notice that $d\alpha$ cannot be identically zero; otherwise, since $H^{1}(M)=0$, the $1$-form $\alpha$ would be nondegenerate and exact, which is not possible on a closed manifold. Now, let $\cZ^2$ be the vector space of 2-currents on $M$. Suppose there is no such $T$, thus $d\alpha \notin \cY$. Again, by a standard application of the Hahn-Banach theorem, there exists a 2-current $c\in \cZ^2$ that is positive on $d\alpha$ and whose kernel contains $\cY$. Accordingly, we have that $c(i_Y \mu)=0$. We claim that $\partial c \in \overline{\cF_Y}$. Indeed, let $\{c_k\}$ be a sequence of $2$-chains that converge (in the weak topology) to the \mbox{$2$-current} $c$. By continuity, it follows that $\int_{c_k} i_Y\mu=:\epsilon_k$ with $\epsilon_k \to 0$ as $k\to \infty$. Take a $2$-chain $b$ such that $\int_b i_Y\mu \neq 0$ (this obviously exists because $Y$ is not identically zero). Then, the sequence of $2$-chains defined as
$$
\widetilde{c_k}:=c_k-\frac{\epsilon_k}{\int_bi_Y\mu}\,b
$$
has zero flux, i.e. $\int_{\widetilde{c_k}}i_Y\mu=0$ and converges in the weak topology to $c$. The continuity of the boundary operator implies that $\partial \widetilde{c_k}$ converges to $\partial c$, thus proving the claim. Finally, by Equation~\eqref{2} we have that $\partial c (\alpha)=c(d \alpha)=0$, which contradicts the fact that $c$ is positive on $d\alpha$. So $d \alpha \in \cY$, as we wanted to show. This completes the proof of the theorem.

\begin{remark}\label{itoiii}
The implication $(i) \Rightarrow (iii)$ holds on any closed $3$-manifold (no need to assume that $H^1(M)=0$).
\end{remark}

\begin{remark}\label{R:vortex}
From the proof of $(ii) \Leftrightarrow (iii)$ we also obtain a characterization of the set of vorticities of a given vector field $X$. More precisely, let $X$ be a non-vanishing vector field on a closed 3-manifold $M$ with volume form $\mu$ and assume that $M$ is not a fibration over the circle (so that, by Tischler's theorem~\cite{Tis}, there do not exist nondegenerate closed $1$-forms). Then, a vector field $Y$ can be written as $Y=T\curl X$ for some metric $g$ and nonzero constant $T$ if and only if $\overline{\cF_Y}\cap\cC_X =\{0\} $.
\end{remark}

\begin{remark}\label{R:curr}
It follows from the proof of the implication $(iii) \Rightarrow (ii)$ that if $c$ is a zero-flux $2$-current for a vector field $Y$, i.e. $c(i_Y\mu)=0$, then $\partial c\in \overline{\cF_Y}$.
\end{remark}

\begin{remark}\label{R:Su}
In the particular case that $Y=X$, it is enough to assume that $\overline{\cB_X}\cap \cC_X =\{0\}$, where $\cB_X$ is the set of boundaries of tangent $2$-chains, i.e.
$$
\cB_X =\{\partial c \,|\, c \, \text{is a 2-chain tangent to $X$}\}\,.
$$
Indeed, Hahn-Banach theorem implies that there exists a $1$-form $\alpha$ such that $\alpha(X)>0$ and $i_Xd\alpha=0$. Since $X$ is non-vanishing, it then follows that $d\alpha=Fi_X\mu$ for some function $F:M\to \RR$. It is easy to check that $F$ is a first integral of $X$ because $L_X\mu=0$. We can then define a vector field $Y:=FX$ that commutes with $X$, i.e. $[X,Y]=0$, and satisfies the zero-flux condition in item~$(iii)$. Applying then Theorem~\ref{T:main} we conclude that $X$ is an Eulerisable flow with constant Bernoulli function (because $i_Xd\alpha=0$), so it is geodesible. This is consistent with Sullivan's characterization~\cite{Sullivan2} of geodesible volume-preserving fields.
\end{remark}

\begin{remark}
According to Remark~\ref{R:vortex}, if $M$ is not a fibration over the
circle, the assumption $\overline{\cF_Y}\cap \cC_X =\{0\}$ implies
that the $1$-form $\alpha$ constructed in the proof of the implication
$(iii) \Rightarrow (ii)$ satisfies $d\alpha=Ti_Y\mu$ for some nonzero
constant $T$; in particular, $\diver Y=0$ and $\overline{\cB_Y}\cap
\cC_X =\{0\}$ (see Remark~\ref{R:Su} for a definition of $\cB_Y$). We
believe that, in general, the existence of a commuting vector field
$Y$ such that $\diver Y=0$ and $\overline{\cB_Y}\cap \cC_X =\{0\}$
does not imply the existence of another commuting field $Y'$ with
$\overline{\cF_{Y'}}\cap \cC_X =\{0\}$; this would show that, contrary
to Sullivan's characterization of geodesible flows, there is no hope
to characterize the Eulerisable flows using only commuting tangent
homologies (instead of commuting zero-flux homologies). This is
supported by the fact that $\diver Y=0$ and $\overline{\cB_Y}\cap
\cC_X =\{0\}$ imply that $d\alpha=Fi_Y\mu$ for some function $F$
(provided that $Y$ is non-vanishing); in the particular case that
$Y=X$, this is a characterization of geodesible volume preserving or stable Hamiltonian flows, which are not contact in general, i.e. the function $F$ is genuinely not constant~\cite[Section~3.9]{CV15}.
\end{remark}

\begin{remark}
We observe that the condition $\overline{\cF_Y}\cap \cC_X =\{0\}$ in item $(iii)$ is independent from the existence of a volume-preserving vector field $Y$ that commutes with $X$. Indeed, using the $C^\infty$ volume-preserving plug introduced by G. Kuperberg~\cite{Ku96} one can construct (in a flow-box) a volume-preserving vector field $X$ that is axisymmetric and contains a Reeb cylinder. The axisymmetry condition is equivalent to saying that there is a volume-preserving vector field $Y$ that commutes with $X$, $[X,Y]=0$. Since Euler flows cannot exhibit plugs (see Section~\ref{S:plugs}), we conclude that the zero-flux homology condition is independent from the existence of $Y$.
\end{remark}

\begin{remark}\label{R:Tao}
Following Tao~\cite{Tao}, we say that a non-vanishing vector field $X$ admits a strongly adapted $1$-form $\alpha$ if $\alpha(X)>0$ and $i_Xd\alpha$ is exact. If $X$ is volume-preserving, the proof of Theorem~\ref{T:main} shows that this is equivalent to being Eulerisable (no need to assume that $H^1(M)=0$). If $X$ is not assumed to preserve a volume form, a simple variation of the proof of Theorem~\ref{T:main} allows one to prove the following: Let $M$ be a closed $3$-manifold with $H^1(M)=0$ endowed with a volume form $\mu$; then $X$ admits a strongly adapted $1$-form $\alpha$ if and only if there exists a (non-identically zero) vector field $Y$ satisfying $[X,Y]=-(\diver X)\,Y$  (i.e., $L_X i_Y \mu=0$) and such that no sequence of elements in $\cF_Y$ can approximate a non-trivial foliation cycle of $X$, that is, $\overline{\cF_Y}\cap\cC_X =\{0\}$. Given $X$ and a strongly adapted $1$-form $\alpha$, the vector field $Y$ is simply defined as $i_Y\mu=d\alpha$. We refer to Remark~\ref{R:TaoN} in Section~\ref{S:N} for a generalization of this result to arbitrary dimensions.
\end{remark}

\section{Proof of Theorem~\ref{T:reeb}}\label{S:Reeb}

We recall that a contact form on a 3-manifold $M$ is a $1$-form $\alpha$ whose associated 3-form $\alpha \wedge d \alpha$ is a volume form (the 2-plane distribution $\xi:= \text{ker } \alpha$ is then everywhere non-integrable, and is called a contact structure). The unique vector field $X$ satisfying $\alpha(X)=1$ and $i_X d \alpha=0$ is called the Reeb field of the contact form $\alpha$.

Let $X$ be the Reeb vector field of a contact form $\alpha$. In particular, $X$ preserves the volume form $\mu:= \alpha \wedge d \alpha$, and we have that $i_X \mu= d\alpha$. Suppose there is a non-zero element $b \in  \overline{\cF_X}\cap \cC_X$, where we define $\cF_{X}$ using the volume $\mu$. Then, on the one hand $b(\alpha)>0$, because $b$ is a foliation cycle, but on the other hand $b$ is in $\overline{\cF_X}$, so $b=\partial c$ where $c$ is a $2$-current that satisfies
\[
c(i_X\mu)=c(d\alpha)=b(\alpha)=0\,,
\]
and we get a contradiction. So if $X$ is Reeb, $\overline{\cF_X}\cap \cC_X =\{0\}$.

Now suppose that $\overline{\cF_X}\cap \cC_X =\{0\}$. Arguing as in the proof of the implication $(iii)\Rightarrow (ii)$ in Theorem \ref{T:main}, we construct a $1$-form $\alpha$ satisfying $\alpha(X)>0$ and $d \alpha= T i_X \mu$, for some constant $T$. If $M$ does not fiber over the circle, we must have that $T \neq 0$. But then $\alpha \wedge d \alpha=T \alpha(X) \mu$ is a volume form, so $\alpha$ is a contact form and the rescaled vector field
\[
X':=\frac{1}{\alpha(X)} X
\]
is its Reeb field.

\section{An application: vector fields constructed with plugs are not Eulerisable}\label{S:plugs}

\subsection{Plugs and geodesible flows}\label{S:geod}

In this subsection we introduce the notion of a plug and we show that vector fields constructed with plugs are not geodesible. This implies, in particular, that vector fields with plugs cannot be Beltrami flows; the general case of steady Euler flows will be considered in the next subsection. Plugs were introduced by Wilson~\cite{Wilson} in the context of the Seifert
conjecture. We start with the definition of a plug:

\begin{definition}\label{defn-plug}
A plug is a 3-manifold $P$ with boundary of the form $D\times [-1,1]$,
where $D$ is a compact surface with boundary (usually a disk). $P$ is endowed with a non-vanishing
vector field $X$, such that
\begin{enumerate}
\item $X$ is vertical in a neighborhood of $\partial P$, that is
  $X=\frac{\partial}{\partial z}$, $z\in [-1,1]$. Thus $X$ is inward
  transverse along $D\times \{-1\}$, outward transverse at $D\times \{1\}$ and tangent to the rest of
$\partial P$.

\item There is a point $p\in D\times \{-1\}$ whose positive orbit is
  trapped in $P$. The set $D_{-1}:=D\times \{-1\}$ is called the entry region of the plug.
\item If the orbit of  $q=(x,-1)\in D\times \{-1\}$ is not trapped, then it
  intersects $D\times \{1\}$ at the point $\overline{q}=(x,1)$. We say that $\overline{q}$ is the point facing $q$, and $D_1:=D\times \{1\}$ denotes the exit region of the plug.
\item There is an embedding of $P$ into $\mathbb{R}^3$ preserving the vertical
  direction.
\end{enumerate}
\end{definition}

A plug allows one to change a vector field on a 3-manifold
locally: given a flow-box, the interior can be replaced by the plug,
thus changing the dynamics. For example, the trapped orbits will now limit
to an invariant set contained inside the plug.

Our main result in this subsection is the proof that a plug cannot be geodesible, and hence
any vector field constructed using a plug is not geodesible. For this we use Sullivan's characterization of geodesible fields~\cite{Sullivan2}. The following
proof is taken from~\cite{Anaphd}. We remark that in this subsection, the vector field $X$ is not assumed to be volume-preserving.

\begin{proposition}\label{prop-plugnong}
The vector field inside a plug $(P,X)$ is not geodesible.
\end{proposition}

\begin{proof}
Let $x\in D_{-1}$ be a point with trapped forward orbit and
assume there is a finite-length curve $\sigma\subset D_{-1}$ from $x$ to $\partial D_{-1}$
such that the orbits of the points $\sigma \setminus \{x\}$ are not
trapped by the plug. Such a point always exists, since the points
in $\partial D_{-1}$ are not trapped. Let $\sigma:[0,1]\to D_{-1}$ be a
parametrisation such that $\sigma(1)=x$ and $\sigma(0)\in \partial D_{-1}$.

We want to show that $X$ is not geodesible. Using Sullivan's theorem~\cite{Sullivan2} we know that it is enough to find a sequence of tangent 2-chains whose boundaries are arbitrarily close to a foliation cycle. Consider the curve $\sigma_t=\sigma([0,t])$, for $t\in [0,1]$. For $t<1$, the orbits of the points in $\sigma_t$ under the flow of $X$ hit $D_1$ after a finite time. Let $A_t$ be the tangent surface defined by the union of the flow-lines of the points in $\sigma_t$, which lie between $D_{-1}$ and $D_1$, i.e.
$$
A_t:=\bigcup_{s=0}^t\gamma(\sigma(s))\,,
$$
where $\gamma(\sigma(s))$ is the $X$-orbit of $\sigma(s)$ inside the plug. Observe that $\widetilde{\sigma_t}:=A_t\cap D_1$ is the curve facing $\sigma_t$, by the exit-entry condition on plugs (item~(3) in Definition~\ref{defn-plug}). Hence we can define $\widetilde{\sigma_1}\subset D_1$ and for every $t\in [0,1]$ we have that $|\sigma_t|=|\widetilde{\sigma_t}|$, where by $|\cdot|$ we will denote the length of the curves and, more generally, the mass of currents (see e.g \cite{Morgan}).

Consider now a sequence $\{t_n\}_{n\in\mN}$ that converges to $1$. Since the orbit of $\sigma(1)=x$ is trapped, and $X$ is non-vanishing, the length of the curve $\gamma(\sigma(t_n))$ goes to infinity as $n\to\infty$, and we can assume without loss of generality that $|\gamma(\sigma(t_m))|\leq |\gamma(\sigma(t_n))|$ for $m \leq n$. Define the sequence of 2-currents
$$
\frac{1}{|\gamma(\sigma(t_n))|}A_{t_n}(\lambda):=\frac{1}{|\gamma(\sigma(t_n))|}\int_{A_{t_n}}\lambda\,,
$$
where $\lambda$ is any $2$-form. We first observe that this sequence of $2$-currents has finite mass. Indeed,
$$
\frac{1}{|\gamma(\sigma(t_n))|}\Big|\int_{A_{t_n}}\lambda\Big|\leq \frac{|A_{t_n}|}{|\gamma(\sigma(t_n))|}\|\lambda\|_{L^\infty}\leq C\|\lambda\|_{L^\infty}\,,
$$
where we have used that $|A_{t_n}|\leq |\gamma(\sigma(t_n))|\cdot|\sigma_{t_n}|\leq C|\gamma(\sigma(t_n))|$, for some constant that does not depend on $n$. This last inequality comes from the assumption that the curve $\sigma$ has finite length.

Moreover, it is clear that the 2-currents $\frac{1}{|\gamma(\sigma(t_n))|}A_{t_n}$ form a sequence of tangent 2-chains. In the following lemma (Lemma \ref{L:limit}) we prove that the boundaries of these 2-chains approach a foliation cycle, thus implying that $X$ cannot be geodesible. The proposition then follows.
\end{proof}

We need to introduce some notation for Lemma~\ref{L:limit}. First, observe that the 2-currents we are considering are normal currents, that is compactly supported currents which have finite mass and boundaries of finite mass as well. In the set of normal currents we consider the flat norm
$$
F(S):=\inf\{|A|+|B|: S=A+\partial B\}\,,
$$
where $S$, $A$ and $B$ are normal currents (for more details see~\cite{Morgan}) and $|\cdot|$ denotes the mass (as defined at the end of the Appendix~\ref{S:appendix}). The set of normal currents is not closed under this norm, however it is easy to check that the flat convergence of normal currents implies the usual weak convergence of currents.

\begin{lemma}\label{L:limit}
$\lim_{n\to \infty}\frac{1}{|\gamma(\sigma(t_n))|}\partial A_{t_n}$ is a non-trivial foliation cycle.
\end{lemma}

\begin{proof}

Consider the sequence of foliation $1$-currents $\frac{1}{|\gamma(\sigma(t_n))|}\gamma(\sigma(t_n))$, we then have that
$$
\Big|\frac{1}{|\gamma(\sigma(t_n))|}\partial A_{t_n}-\frac{1}{|\gamma(\sigma(t_n))|}\gamma(\sigma(t_n))\Big|\leq \frac{1}{|\gamma(\sigma(t_n))|}\Big[|\sigma_1|+|\widetilde{\sigma_1}|+|\gamma(\sigma(0))|\Big]\,,
$$
so the difference converges in the flat norm (and in the weak topology) to zero as $n\to\infty$. Additionally, the flat norm of the $1$-currents $\frac{1}{|\gamma(\sigma(t_n))|}\gamma(\sigma(t_n))$ is less or equal to one, because they have mass one. Since the space of $1$-currents is Montel, there is a convergent subsequence $\frac{1}{|\gamma(\sigma(t_{n_k}))|}\gamma(\sigma(t_{n_k}))$. Hence, the limit defines the $1$-current with mass one (and hence non-trivial):
$$
S:=\lim_{k\to \infty}\frac{1}{|\gamma(\sigma(t_{n_k}))|}\partial A_{t_{n_k}}=\lim_{k\to \infty}\frac{1}{|\gamma(\sigma(t_{n_k}))|}\gamma(\sigma(t_{n_k}))\,.
$$
Since the boundary operator $\partial$ is continuous, it follows that $S$ is a cycle. Moreover, since the space of foliation currents is a closed convex cone $\mathcal{C}_X$ containing the sequence $\frac{1}{|\gamma(\sigma(t_{n_k}))|}\gamma(\sigma(t_{n_k}))$, it contains its limit. Thus $S$ is a foliation cycle, as we wanted to prove.
\end{proof}

\begin{remark} The two important properties of a plug that are used in
  the proof above is that there are trapped orbits and that the map
  from the entry to the exit is absolutely continuous, thus mapping
  curves of bounded length onto curves of bounded length.
\end{remark}

\subsection{Plugs are not Eulerisable}

In this subsection we show that vector fields that are constructed using plugs are not Eulerisable. This implies, in particular, that Euler flows cannot contain Wilson-type plugs (i.e. with Reeb cylinders). We first recall that Eulerisable fields with constant Bernoulli function, i.e. Beltrami flows, are geodesible (because $\alpha(X)>0$ and $i_Xd\alpha=0$), and hence by Proposition~\ref{prop-plugnong} they cannot be constructed using plugs. Accordingly, key to prove that plugs are not Eulerisable is to analyze the Euler flows with non-constant Bernoulli function. In this case, Sullivan's theorem~\cite{Sullivan2} implies that $\overline{\cB_X}\cap \cC_X$ contains non-trivial elements (see Remark~\ref{R:Su} for the definition of $\cB_X$). The following lemma is an instrumental tool to prove the main theorem of this subsection. In the statement we denote by $G$ the set of critical points of the Bernoulli function $B$ of the Euler flow, i.e.
$$G:=\{x\in M: dB(x)=0\}\,.$$

\begin{lemma}\label{prop-Eulernong}
Let $X$ be a non-vanishing Euler flow that is not geodesible. Let $z\neq 0$ be a foliation cycle in $\overline{\cB_X}\cap \cC_X$, and let $c_n$ be a sequence of 2-chains tangent to $X$ that converge to a tangent $2$-current $A$ such that $\partial A=z$. Then the support of $A$ satisfies $supp(A)\cap G^c\neq \emptyset$.
\end{lemma}
\begin{proof}
Since $X$ is not geodesible, its Bernoulli function is not constant, so the complement $G^c$ is not empty. Suppose that the support of $A$ is contained in $G$. Since $X\times \curl X=0$ in $G$, and $X$ is non-vanishing, then $\curl X$ is either zero or collinear with $X$ on $G$. Recall that $\curl X$ is defined as $i_{\curl X}\mu=d\alpha$, where $\alpha$ is the $1$-form dual to $X$. Accordingly, $0=A(d\alpha)=\partial A(\alpha)= z(\alpha)>0$, which is a contradiction. 
\end{proof}


We are now ready to prove that the insertion of a plug $(P,X)$ is not Eulerisable. Observe that to have any hope that plug insertions
can be done in the Eulerisable category, the vector field $X$ has to
preserve volume. Thus we assume that $(P,X)$ is a volume preserving
plug. In this case the trapped set of $P$ has empty interior. In the proof of the following theorem, we shall use the notation introduced in Subsection~\ref{S:geod} without further mention.

\begin{theorem}\label{T:plugs}
The vector field inside a plug $(P,X)$ is not Eulerisable.
\end{theorem}

\begin{proof}
As explained in the proof of Proposition~\ref{prop-plugnong}, the mass of the currents $\frac{1}{|\gamma(\sigma(t))|}A_{t}$ is bounded by the length of $\sigma_{t}$; then, as the space of \mbox{$2$-currents} is Montel, we can subtract a convergent subsequence $\frac{1}{|\gamma(\sigma(t_n))|}A_{t_n}$ (in the weak topology). Let $A$ be the limit 2-current. First observe that it is non trivial because $\frac{1}{|\gamma(\sigma(t_n))|}\partial A_{t_n}$ converges to a non-zero foliation cycle $\partial A$ of mass one (c.f. Lemma~\ref{L:limit}). Now observe that for any $2$-form $\omega$ whose kernel contains $X$, we have that $A_{t_n}(\omega)=0$, thus implying that $A(\omega)=0$ by continuity, so $A$ is a $2$-current tangent to $X$. Finally, since $A$ has compact support, and both the mass of $A$ and the mass of its boundary are bounded, $A$ is a normal $2$-current.

Assume now that the vector field $X$ of the plug is Eulerisable. This Euler vector field has a Bernoulli function $B$, which we assume to be non constant. Otherwise, the field $X$ is geodesible and the result follows from Proposition~\ref{prop-plugnong}. The following Lemma \ref{L.flux} shows that the vector field $\curl X$ has zero flux through $A$, i.e. $A(i_{\curl X}\mu)=0$. By Remark~\ref{R:curr} we have that $\partial A\in \overline{\cF_{\curl X}}$, but $\partial A$ is a non trivial foliation cycle and $[X,\curl X]=0$, which is a contradiction according to Theorem~\ref{T:main} (and Remark~\ref{itoiii}). We conclude that $X$ cannot be a steady Euler flow.

\end{proof}

\begin{lemma}\label{L.flux}
The $2$-current $A$ satisfies that $A(i_{\curl X}\mu)=0$.
\end{lemma}

\begin{proof}

We claim that the quantity
\[
\frac{1}{|\gamma(\sigma(t_n))|}\Big|A_{t_n}(i_{\curl X}\mu)\Big|
\]
tends to zero as $t_n \rightarrow 1$, which implies that $A(i_{\curl X}\mu)=0$ by continuity.

To see this, first observe that the Euler equation implies that
\[
i_{\curl X}\mu=\frac{1}{\alpha(X)}(X \cdot \curl X) i_X \mu-\frac{1}{\alpha(X)} \alpha \wedge dB\,,
\]
where we recall that $\alpha=i_X g$. This identity is proved using that $\alpha\wedge d\alpha=(X\cdot \curl X)\mu$, and
$$
i_X(\alpha\wedge i_{\curl X}\mu)=\alpha(X) i_{\curl X}\mu+\alpha\wedge dB\,.
$$

Therefore, since the 2-chains $A_{t_n}$ are tangent to $X$, we have that
\[
\frac{1}{|\gamma(\sigma(t_n))|} \int_{A_{t_n}} i_{\curl X}\mu=\frac{-1}{|\gamma(\sigma(t_n))|} \int_{A_{t_n}} \frac{1}{\alpha(X)} \alpha \wedge dB\,.
\]

To compute this integral, we introduce appropriate coordinates $(s,\tau)$ on the surface $A_{t_n}$. Using the flow $\phi_X^s$ of the vector field $X$, any point on $A_{t_n}$ can be described as $\phi_X^{s}(\sigma(\tau))$, with $\tau\in (0,t_n)$ and $s\in (0,s_0(\tau))$, where $s_0(\tau)$ is the time that takes the orbit $\gamma(\sigma(\tau))$ of $X$ to go from the entry $D_{-1}$ to the exit $D_1$. Accordingly, using the holonomic basis of fields $\{\partial_s,\partial_\tau\}$, noticing that any $1$-form $\beta$ can be written as $\beta=\beta(\partial_s)ds +\beta(\partial_\tau)d\tau$, and that $X=\partial_s$ in these coordinates, the integral above can be written as
\[
\frac{1}{|\gamma(\sigma(t_n))|} \int_{0}^{t_n}  \int_{\gamma(\sigma(\tau))} \frac{1}{\alpha(X)} (dB(X)\alpha(\partial_{\tau})-dB(\partial_{\tau}) \alpha(X))\, ds\, d\tau\,.
\]
Now, since $B$ is a first integral of $X$, we have that $dB(X)=0$ and $dB(\partial_{\tau})$ does not depend on $s$, so we readily get
\[
\frac{1}{|\gamma(\sigma(t_n))|} \int_{A_{t_n}} i_{\curl X}\mu=\frac{-1}{|\gamma(\sigma(t_n))|} \int_{0}^{t_n}  dB(\partial_{\tau})\Big( \int_{\gamma(\sigma(\tau))} ds \Big) d\tau \,.
\]

For any point $t_l\in[0, 1)$, there is a point $t^*_l\in [0,t_l]$ such that
\[
|\gamma(\sigma(t^{*}_{l}))|:=\underset{\tau \in [0, t_l]}{\text{sup }} |\gamma(\sigma(\tau))|\,.
\]

Taking $n$ large enough, we can safely assume that $t_n>t_l$ and $|\gamma(\sigma(t_n))|>|\gamma(\sigma(t_l))|$. This can be used to write
\begin{align*}
\frac{1}{|\gamma(\sigma(t_n))|} \Big|\int_{A_{t_n}} i_{\curl X}\mu\Big|=\frac{1}{|\gamma(\sigma(t_n))|}\Bigg|\int_{0}^{t_l}  dB(\partial_{\tau}) \Big( \int_{\gamma(\sigma(\tau))} ds \Big) d\tau+\\
\int_{t_l}^{t_n}  dB(\partial_{\tau}) \Big( \int_{\gamma(\sigma(\tau))} ds \Big) d\tau\Bigg| \leq \\\frac{1}{\text{min}_M|X|}\Bigg ( \frac{ |\gamma(\sigma(t^{*}_{l}))|}{|\gamma(\sigma(t_n))|}|B(\sigma(t_l))-B(\sigma(0))|+|B(\sigma(t_n)-B(\sigma(t_l))|\Bigg)\,,
\end{align*}
where the minimum value of $|X|$ on $M$ is positive because $X$ is non-vanishing.

We claim that we can make the quantity in the right hand side of the previous bound as small as we wish for $n$ large enough. Indeed, for any $\epsilon>0$ there is $t_l \in (0, 1)$ close enough to $1$ such that $|B(\sigma(t_l))-B(\sigma(1))|\leq \epsilon$, and there are infinitely many $t_n\in (t_l, 1)$, such that, on the one hand, $|B(\sigma(t_n))-B(\sigma(1))|\leq \epsilon$, and on the other hand
\[
\frac{|\gamma(\sigma(t^{*}_{l}))|}{|\gamma(\sigma(t_n))|} \leq \epsilon
\]
because $|\gamma(\sigma(t_n))|\to \infty$ as $n\to \infty$. Hence, for $n$ large enough,
\[
\frac{1}{|\gamma(\sigma(t_n))|} \Bigg|\int_{A_{t_n}} i_{\curl X}\mu\Bigg| \leq C\epsilon|B(\sigma(t_l))-B(\sigma(0))|+2 \epsilon\leq C\epsilon\,,
\]
where $C$ is a constant that depends on $X$ and $B$, but not on $\epsilon$.  Taking the limit $n\to \infty$ we infer that $|A(i_{\curl X}\mu)|\leq C\epsilon$ for any $\epsilon>0$, thus proving the lemma.

\end{proof}

\begin{remark}\label{R:Taoplugs}
A simple variation of the proof of Theorem~\ref{T:plugs} shows that a non-vanishing vector field $X$ (not necessarily volume-preserving) with a strongly adapted $1$-form $\alpha$ (see Remark~\ref{R:Tao} and Theorem~\ref{T:TaoplugsN}) cannot be constructed inserting plugs.
\end{remark}

\begin{remark}
Theorem~\ref{T:plugs} implies that an Eulerisable flow cannot contain plugs with Reeb components. In fact it is ready to prove that, in general, a steady Euler flow cannot exhibit Reeb cylinders. A straightforward proof using Stokes theorem is presented in~\cite[Lemma 2.3]{CV17}; it can also be derived from Theorem~\ref{T:main} by constructing a sequence of $2$-chains tangent to $\curl X$ (on the Reeb cylinder) such that their boundaries converge to a foliated cycle of $X$.
\end{remark}

\section{Higher dimensional Eulerisable flows}\label{S:N}

A steady Euler flow on an $n$-dimensional Riemannian manifold $(M,g)$ with volume form $\mu$ is a solution to the stationary Euler equations:
\[
i_X d \alpha=-dB\,,\,\,\,\, L_X \mu=0\,,
\]
where $\alpha:=i_X g$ is the $1$-form dual to $X$. Analogously to Definition~\ref{D:Euler}, an \emph{Eulerisable flow} in an $n$-dimensional manifold is a volume-preserving vector field that solves the stationary Euler equations for some Riemannian metric.

The $3$-dimensional case is special in the sense that the vorticity can be defined as the vector field $\curl X$, i.e. $i_{\curl X}\mu=d\alpha$. In higher dimension $n>3$, the vorticity is just the $2$-form $d \alpha$, whose kernel has dimension $n-2>1$. Keeping this is mind, all the results in previous sections can be readily generalized to closed manifolds of arbitrary dimension.

First, we have the following characterization of Eulerisable flows in terms of differential forms (see also Remark~\ref{R:Tao}):

\begin{lemma}\label{L.forms}
A non-vanishing volume-preserving vector field $X$ on a closed manifold $M$ of dimension $n\geq 2$ is Eulerisable if and only if there exists a $1$-form $\alpha$ such that $\alpha(X)>0$ and $\iota_X d\alpha$ is exact.

\end{lemma}

\begin{proof}
%
The proof is identical to that of the equivalence between items $(i)$ and $(ii)$ in Theorem \ref{T:main}.
\end{proof}

For the $n$-dimensional homological characterization (c.f. Theorem~\ref{T:main}), the role of the vorticity vector field $Y$ is played by a $2$-form. Accordingly, for a $2$-form $\omega$, we introduce the following linear subspace of the space of $1$-currents:
$$
\cF_{\omega} :=\Big\{\partial c \,:\, c \, \text{is a 2-chain with $\int_{c} \omega=0$}\Big\}\,.
$$

\begin{theorem}\label{T:mainN}
A non-vanishing volume-preserving vector field $X$ on a closed $n$-dimensional manifold $M$ ($n\geq 3$) with trivial first cohomology group is Eulerisable
if and only if there exists a (non-identically zero) $2$-form $\omega$ that is preserved by $X$, i.e. $L_{X} \omega=0$, and such that no sequence of elements in $\cF_{\omega}$ can approximate a non-trivial foliation cycle of $X$, that is, $\overline{\cF_{\omega}}\cap
  \cC_X =\{0\}$.

\end{theorem}

\begin{proof}
We argue exactly as in the proof of Theorem~\ref{T:main}, replacing $\cF_Y$ by $\cF_{\omega}$. The ``only if" part follows by taking $\omega:= d i_X g$ (the assumptions $H^1(M)=0$ is not needed). As for the reverse implication, the condition $\overline{\cF_{\omega}}\cap
  \cC_X =\{0\}$ allows us to find, by means of the Hahn-Banach theorem, a $1$-form $\alpha$ satisfying $\alpha(X)>0$ and $d \alpha=T \omega$ for some constant $T$. Since $L_X \omega=0$ by hypothesis, we have that $L_X d \alpha= i_X d \alpha=T L_X \omega=0$. Being $H^{1}(M)=0$, we deduce that $T \neq 0$ and $i_X d \alpha$ is exact, and then we apply Lemma \ref{L.forms} to conclude that $X$ is Eulerisable. Observe that we do not need to assume that the form $\omega$ is closed; it follows a posteriori from the proof.

\end{proof}

\begin{remark}\label{R:TaoN}
In Theorem~\ref{T:mainN}, the assumption that $X$ is volume-preserving is only needed to apply Lemma~\ref{L.forms}. In particular, it is not used when showing the existence of the $1$-form $\alpha$ such that $\alpha(X)>0$ and $i_X d \alpha$ is exact. Accordingly, the proof of Theorem~\ref{T:mainN} also provides a characterization of the $n$-dimensional non-vanishing vector fields $X$ admitting a strongly adapted $1$-form $\alpha$ in the sense of Tao (see Remark~\ref{R:Tao}): such a $1$-form exists if and only if there is a (non-trivial) $2$-form $\omega$ such that $L_X\omega=0$ and $\overline{\cF_{\omega}}\cap
  \cC_X =\{0\}$ (actually, the assumption $H^1(M)=0$ is needed only to prove the ``if'' part).
\end{remark}

Finally, plugs can also be defined in arbitrary dimension, by taking $D$ in Definition~\ref{defn-plug} to be a compact manifold with boundary of dimension $n-1$ (usually a ball) instead of a surface. We then have:

\begin{theorem}\label{T:plugsN}
The vector field inside an $n$-dimensional plug $(P,X)$, $n\geq 3$, is not Eulerisable.
\end{theorem}

\begin{proof}
Following the proof and notation of Theorem~\ref{T:plugs}, we construct a sequence of 2-chains
\[
\frac{1}{|\gamma(\sigma(t))|}A_{t}\,,
\]
which are tangent to $X$ and, moreover, converge to a 2-current $A$ whose boundary $\partial A$ is a foliation cycle of $X$. We now proceed by contradiction: we claim that if $X$ is a steady Euler flow, then $A( d\alpha)=0$, where $\alpha:=i_Xg$, i.e. $\partial A \in \cF_{d \alpha}$. Since the Euler equations imply that $L_Xd\alpha=0$, this yields a contradiction with Theorem~\ref{T:mainN}, so $X$ cannot be Eulerisable.

To show that $A( d \alpha)=0$ we generalize the proof of Lemma~\ref{L.flux} by using the following identity, satisfied by steady Euler flows in arbitrary dimensions:
\[
d \alpha=\frac{1}{\alpha(X)} i_X (\alpha \wedge d \alpha)-\frac{1}{\alpha(X)} \alpha \wedge dB \,.
\]
Since $A_{t_n}$ are surfaces tangent to $X$, we have $A_{t_n}\Big(\frac{1}{\alpha(X)}i_X (\alpha \wedge d \alpha)\Big)=0$ and
\[
\frac{1}{|\gamma(\sigma(t_n))|} \int_{A_{t_n}} d \alpha=\frac{-1}{|\gamma(\sigma(t_n))|} \int_{A_{t_n}} \frac{1}{\alpha(X)} \alpha \wedge dB\,.
\]
Now the argument is the same, mutatis mutandis, as in the proof of Lemma~\ref{L.flux}, thus proving the claim.

Instead of arguing by contradiction, there is an alternative proof that consists in showing that for any $2$-form $\omega$ such that $L_X \omega=0$, we have $\partial A \in \cF_{\omega}$. Indeed, as in the proof of Lemma~\ref{L.flux}, we introduce coordinates $(s,\tau)$ on the surface $A_{t_n}$, with $\tau\in (0,t_n)$ and $s\in (0,s_0(\tau))$, where $s_0(\tau)$ is the exit time of the orbit $\gamma(\sigma(\tau))$. For any $2$-form $\omega$ we can write, in the holonomic basis $\{\partial_s, \partial_{\tau}\}$ of vector fields in $A_{t_n}$,
\[
\frac{1}{|\gamma(\sigma(t_n))|} \int_{A_{t_n}} \omega=\frac{1}{|\gamma(\sigma(t_n))|} \int_{A_{t_n}} \omega_{(s, \tau)}(\partial_s, \partial_{\tau}) ds\, d\tau\,.
\]

Define the function $g(s, \tau):= \omega_{(s, \tau)}(\partial_s, \partial_{\tau})$. Now, the fact that $L_X \omega=0$ and $[\partial_s, \partial_{\tau}]=0$, imply that $\frac{\partial}{\partial s} g(s, \tau)=0$, hence
\[
\frac{1}{|\gamma(\sigma(t_n))|} \int_{A_{t_n}} \omega=\frac{1}{|\gamma(\sigma(t_n))|} \int_{0}^{t_n}  g(\tau) \Big( \int_{\gamma(\sigma(\tau))} ds \Big) d\tau \,.
\]
Since $g(\tau)$ is uniformly bounded, arguing as in Lemma~\ref{L.flux} this integral can be shown to be as small as one wishes, which implies that $A(\omega)=0$. Then ``only if'' part of Theorem~\ref{T:mainN}, which holds on any closed manifold, then implies that $X$ cannot be Eulerisable.
\end{proof}

\begin{remark}\label{R:TaoplugsN}
Observe that in the proof of Theorem \ref{T:plugsN} the volume-preserving hypothesis is not used. This yields the following (see also Remark \ref{R:TaoN}):
\end{remark}

\begin{theorem}\label{T:TaoplugsN}
The vector field inside an $n$-dimensional plug $(P,X)$ cannot admit a strongly adapted $1$-form.
\end{theorem}

\begin{proof}
In view of Theorem \ref{T:plugsN} and Remark \ref{R:TaoplugsN}, the vector field $X$ in a plug (not necessarily volume preserving) has $\overline{\cF_{\omega}}\cap
  \cC_X \neq \{0\}$ for any $2$-form $\omega$ with $L_{X} \omega=0$.  Hence by Remark \ref{R:TaoN}, it cannot admit any strongly adapted $1$-form.

\end{proof}

\appendix
\section{Some background on currents}\label{S:appendix}

To keep the article self-contained, in this section we recall some fundamental facts about currents that we use all along this paper. We assume that the ambient manifold $M$ is closed, otherwise one has to consider forms that have compact support.

Consider the linear space $\Omega^{k}(M)$ of smooth $k$-forms on an $n$-dimensional manifold $M$. We endow this space with the topology induced by the family of norms $||\cdot||_{C^{r}(K)}$, $r \in \mathbb{N}$, $K$ a compact subset of $M$. By definition, this makes $\Omega^{k}(M)$ a locally convex topological vector space.

A $k$-current is a continuous linear functional $c: \Omega^{k}(M) \rightarrow \RR$. We denote the linear space of $k$-currents by $\cZ^{k}(M)$, and endow it with the weak topology, that is the topology generated by the family of semi-norms $||c||_{\alpha}:=|c(\alpha)|,\, \alpha \in \Omega^{k}(M)$. With this topology the space of $k$-currents is clearly a locally convex topological vector space.

Let $\omega$ be any smooth $k$-form. The following are the most important examples of $k$-currents appearing in our arguments:
\begin{enumerate}
\item Dirac $k$-currents: for any point $p\in M$ and any set of linearly independent $k$ vectors $v_1,...,v_k \in T_p M$ we have the $k$-current
\[
\delta^{v_1,...,v_k}_{p}(\omega):= \omega_p(v_1,...,v_k)\,.
\]
\item Singular $k$-chains acting by integration: let $\Sigma$ be a singular $k$-chain, then we have the $k$-current

\[
\Sigma(\omega):=\int_{\Sigma} \omega\,.
\]

\item Smooth $(n-k)$-forms: let $\alpha$ be an $(n-k)$-form, then we have the $k$-current
\[
\alpha(\omega):=\int_M \alpha \wedge \omega\,.
\]

\end{enumerate}

It is well known that the linear subspaces spanned by the currents of each of the above types are dense in the space of currents. The following properties of the space of $k$-currents are easily verified, and key to our arguments:

\begin{enumerate}

\item $\cZ^{k}$ is a Montel space, i.e. the weak closure of a bounded subset is compact (this is also known as the Heine-Borel property). We recall that a bounded subset is a subset that is bounded with respect to every semi-norm (not necessarily in a uniform way).

\item Linear continuous functionals on $\cZ^{k}$ can be identified with smooth $k$-forms (this can be easily seen by considering the action of linear functionals on Dirac $k$-currents and their derivatives).

\item The boundary operator $\partial : \cZ^{k+1}(M) \rightarrow \cZ^{k}(M)$ defined as
\[
\partial c (\omega)=c (d \omega)\,,
\]
for any $\omega\in \Omega^k(M)$, is continuous.

\end{enumerate}

We finish by introducing two important norms on the space of currents, the mass $|\cdot|$ and the flat norm $F$ (which were already defined in Subsection~\ref{S:geod}). For $c\in \cZ^k(M)$ we set:
\begin{eqnarray*}
|c| & := & \sup \{|c(\alpha)|\,:\, \|\alpha\|_{C^0}\leq 1 \}\,,\\
F(c) & := & \min \{|a|+|b| \, : \, c=a +\partial b\}\,.
\end{eqnarray*}
A current $c$ is normal if both $c$ and $\partial c$ have finite mass. For sequences of normal currents with bounded mass and boundaries of bounded mass, weak convergence and convergence in the flat norm are equivalent. We refer to~\cite{Morgan} for more details.

\section*{Acknowledgments}
D.P.-S. was supported by the ERC Starting Grant~335079, and the grants MTM2016-76702-P (MINECO/FEDER) and Europa Excelencia EUR2019-103821 (MCIU), and partially supported by the ICMAT--Severo Ochoa grant SEV-2015-0554. F.T.L acknowledges the financial support and hospitality of the Max Planck Institute for Mathematics. A.R. thanks the IdEx Unistra, Investments for the future program of the French Government.

\bibliographystyle{amsplain}

\end{document}